\documentclass[12pt,twoside]{amsart}
\usepackage{amsmath, amsthm, amscd, amsfonts, amssymb, graphicx}
\usepackage[bookmarksnumbered, plainpages]{hyperref}

\usepackage{mathrsfs,mathtools}
\usepackage{tensor}

\newcommand{\statementtitle}[1]{\noindent\textbf{#1}\vspace{0.5em}\par} 

\newtheorem{thm}{Theorem}[section]

\newtheorem{lem}[thm]{Lemma}
\newtheorem{prop}[thm]{Proposition}

\numberwithin{equation}{section}

\def\tr{\mbox{tr}}

\begin{document}

\title{\bf Monge-Amp\`ere type equation for the Nakano positive curvature tensor of holomorphic vector bundles}
\author{Changpeng Pan}

\address{Changpeng Pan\\School of Mathematics and Statistics\\
Nanjing University of Science and Technology\\
Nanjing, 210094,P.R. China\\ }
\email{mathpcp@njust.edu.cn}


\maketitle

\begin{abstract}
For any Hermitian holomorphic vector bundle with Nakano positive curvature tensor, Demailly introduced a Monge-Amp\`ere type equation. When the rank of the bundle is $1$, it becomes the usual Monge-Amp\`ere equation. In this paper, we solve this equation in the conformal class of a Nakano positive Hermitian metric.

\end{abstract}

\vskip 0.2 true cm


\pagestyle{myheadings}
\markboth{\rightline {\scriptsize C. Pan}}
         {\leftline{\scriptsize Monge-Amp\`ere type equation for the Nakano positive curvature tensor}}

\bigskip
\bigskip


\section{ Introduction}
Let $(X,\omega)$ be a compact K\"ahler manifold. Let $(E,H)$ be a Hermitian holomorphic vector bundle over $X$. Let $\Theta_{H}=\sqrt{-1}D_{H}^{2}$ be the Chern curvature. Define a quadratic form on $TX\otimes E$ by
\begin{equation*}
\widetilde{\Theta}_{H}(\gamma)=\gamma^{i\alpha}\overline{\gamma^{j\beta}}\langle\Theta(\partial_{i},\partial_{\bar{j}})e_{\alpha},e_{\beta}\rangle_{H},\quad\quad \forall\gamma=\gamma^{i\alpha}\partial_{i}\otimes e_{\alpha}.
\end{equation*}
$(E,H)$ is called Nakano positive if the quadratic form $\widetilde{\Theta}$ induced by Chern curvature is positive. Similarly, one can define a quadratic form $^{T}{\widetilde{\Theta}}_{H}$ on $TX\otimes E^{*}$, $(E,H)$ is called dual Nakano positive if it is positive. There is also a slightly weaker condition known as Griffiths positivity (see \cite{De21}). These three positivities can all imply the ampleness of a vector bundle. However, conversely, examples \cite{De22} can be found where a vector bundle is ample but not (dual) Nakano positive. It is conjectured by Griffiths \cite{G} that ampleness is equivalent to Griffiths positivity. So far, it is only known that the Griffiths conjecture holds true in the case of line bundles \cite{U} or when the base manifold is a Riemann surface \cite{CF}.

When $\Theta_{H}$ is Nakano positive, one can define a positive $(n,n)$ form by
\begin{equation*}
{\det}_{TX\otimes E}(\Theta_{H})^{\frac{1}{r}}=\det({\Theta_{i\alpha\bar{j}}}^{\beta})_{(i\alpha),(j\beta)}^{\frac{1}{r}}(\sqrt{-1})^{n}dz^{1}\wedge d\bar{z}^{1}\cdots \wedge dz^{n}\wedge d\bar{z}^{n}.
\end{equation*}
It is well-defined, i.e., it neither depends on the choice of coordinates nor on the
choice of trivialisation for $E$. Demailly \cite{De21,De22} proposed a program to prove the Griffiths conjecture, which includes solving a new system of elliptic differential equations as following
\begin{align}
&{\det}_{TX\otimes E}(\Theta_{H}+t\tr_{E}(\Theta_{H})\otimes Id_{E})^{\frac{1}{r}}=f_{t}\omega^{n}\label{DE1},\\
&(\Theta_{H}-\frac{1}{r}\tr_{E}(\Theta_{H})\otimes Id_{E})\wedge\omega^{n-1}=g_{t},\label{DE2}
\end{align}
for $t\in(-\frac{1}{r},1]$, where $f_{t}$ is a function depends on lower-order term of $H$ and $g_{t}$ is a trace-free Hermitian Endomorphism of $E$ also depends on lower-order term of $H$. In the context of solving this system, Pingali\cite{Pin} and Mandal \cite{Ma} have each made contributions in certain simplified scenarios. However, fully resolving this system remains a significant challenge.

The equation (\ref{DE2}) is of the Hermitian-Yang-Mills type, and when $g_{t}=-\epsilon\log(K^{-1}H)\otimes\omega^{n}$, as shown by the work of Uhlenbeck-Yau \cite{UY}, it has a smooth solution. When $E$ is polystable, then we can take $\epsilon=0$. The left-hand side of the equation (\ref{DE2}) is conformally invariant. If we also take $g_{t}$ to be conformally invariant, then any solution that exists will have all its conformal metrics as solutions to the equation. If we consider equation (\ref{DE1}), then the previous discussion seems to suggest that we should solve it in the conformal class of a fixed metric.

Now let $(E,K)$ be a rank $r$ Hermitian holomorphic vector bundle which is Nakano positive. We consider the following equation
\begin{equation*}
{\det}_{TX\otimes E}(\Theta_{H})^{\frac{1}{r}}=e^{\phi}(\frac{\det K}{\det H})^{\lambda}\omega^{n},
\end{equation*}
where $\omega=\frac{1}{r}\tr_{E}(\Theta_{K})>0$, $\phi\in C^{\infty}(X)$ is a fixed function and $\lambda\geq 0$. Let
\begin{equation*}
\mathcal{H}_{K}=\{u\in C^{\infty}(X,\mathbb{R})\mid  \Theta_{K}+\sqrt{-1}\partial\bar{\partial}u\otimes Id_{E}>_{N}0\},
\end{equation*}
then $H=e^{-u}K$ is a Nakano positive metric. 
Considering the above equation in the conformal class of $K$, it becomes
\begin{equation}\label{eq:MA}
{\det}_{TX\otimes E}(\Theta_{K}+\sqrt{-1}\partial\bar{\partial}u\otimes Id_{E})^{\frac{1}{r}}=e^{\lambda r u+\phi}\omega^{n}.
\end{equation}
When $r=1$, it corresponds to the Calabi-Yau theorem, whose existence of solutions was resolved by Aubin\cite{Au} and Yau\cite{Yau}. When $r>1$, the situation becomes more complicated. Even without considering the background of this problem, from the perspective of the equation alone, the existence of its solutions remains a question worth exploring.

In this paper, we study this problem and obtain the following theorem.
\begin{thm}\label{thm1}
Let $X$ be a K\"ahler manifold and $(E,K)$ be a rank-$r$ Hermitian holomorphic vector bundle over $X$. Let $\Theta_{K}$ be Nakano positive. 
\begin{itemize}
\item[(1)] If $\lambda>0$, there is a unique smooth function $u\in\mathcal{H}_{K}$ that satisfies (\ref{eq:MA});
\item[(2)] If $\lambda=0$, there is a unique constant $c$ and a unique smooth function $u\in\mathcal{H}_{K}$ that satisfies $\sup_{X}u=0$ such that
\begin{equation*}
{\det}_{TX\otimes E}(\Theta_{K}+\sqrt{-1}\partial\bar{\partial}u\otimes Id_{E})^{\frac{1}{r}}=e^{\phi+c}\omega^{n}.
\end{equation*}
\end{itemize}
\end{thm}
To prove Theorem \ref{thm1}, we employ the method of continuity. The key to the proof is to derive the $C^{0}$, $C^{1}$ and $C^{2}$ estimates, with the most challenging part being the $C^{2}$ estimate. The reason why the $C^{2}$ estimate is difficult is that after taking the second derivative of the equation and separating out the good third-order term, a zero-order term with a bad coefficient appears. After obtaining the above estimates, we can use Evans-Krylov's method \cite{Tr} to get the $C^{2,\alpha}$ estimate, and then improve the regularity of $u$ to $C^{\infty}$ using the Schauder estimates for elliptic equations.

Recently, George \cite{Ge} studied the solutions of the Monge-Amp\`ere equation for the $(p,p)$-form. Formally, the equations we study are very similar. However, our approach to handling the second-order estimate is quite different.

This paper is organized as follows. In Section \ref{sec:Pre}, we provide some basic properties about equation (\ref{eq:MA}) and prove its ellipticity. In Section \ref{sec:C0}, we provide the $C^{0}$ estimate. In Section \ref{sec:C2}, we provide the $C^{2}$ estimate which includes the gradient term. In Section \ref{sec:C1}, we provide the $C^{1}$ estimate based on Liouville theorem. In Section \ref{sec:EK}, we establish the $C^{2,\alpha}$ estimate based on the complex version of the Evans-Krylov estimate proposed by Trudinger.

\section{Preliminaries}\label{sec:Pre}
 First, we transform the equation into
\begin{equation*}
L(u):=\log\frac{\det_{TX\otimes E}(\Theta_{K}+\sqrt{-1}\partial\bar{\partial}u\otimes Id_{E})^{\frac{1}{r}}}{\omega^{n}}-\lambda r u=\phi.
\end{equation*}
For simplicity, let $\Theta_{0}=\Theta_{K}$ and $\Theta_{u}=\Theta_{K}+\sqrt{-1}\partial\bar{\partial}u\otimes Id_{E}$. Furthermore, we do not distinguish between $\Theta$ and the quadratic form $\widetilde{\Theta}$ induced by it. Let
\begin{equation*}
\Theta=\Theta_{i\alpha\bar{j\beta}}dz^{i}\otimes e^{\alpha}\otimes d\bar{z}^{j}\otimes \overline{e^{\beta}}, \quad \Theta^{-1}=\Theta^{i\alpha\bar{j\beta}}\partial_{i}\otimes e_{\alpha}\otimes \partial_{\bar{j}}\otimes \overline{e_{\beta}},
\end{equation*}
where $\Theta^{-1}$ is the inverse quadratic form on $T^{*}X\otimes E^{*}$. Take $\{e_{\alpha}\}$ be an orthonormal basis with respect to the metric $K$.
Then
\begin{equation*}
\begin{split}
\delta L(u)(v)=&\frac{1}{r}\sum_{\alpha=1}^{r}\Theta_{u}^{i\alpha\overline{j\alpha}}\partial_{i}\partial_{\bar{j}}v-\lambda rv\\
=&\frac{1}{r}\alpha^{i\bar{j}}_{u}\partial_{i}\bar{\partial}_{j}v-\lambda rv
\end{split}
\end{equation*}
is elliptic, since $\alpha^{i\bar{j}}_{u}=\sum_{\alpha}\Theta_{u}^{i\alpha\overline{j\alpha}}$ is positive.

\subsection{Continuity method and openness: $\lambda>0$}
When $\lambda>0$, let $\phi_{1}=\log\frac{\det_{TX\otimes E}(\Theta_{K})^{\frac{1}{r}}}{\omega^{n}}$ and $\phi_{t}=\phi+t(\phi_{1}-\phi)$. Then $L(u)=\phi_{t}$ is solvable at $t=1$. 

We can show that $\delta L(u):C^{k+2,\alpha}(X)\to C^{k,\alpha}(X)$ is an isomorphism. Since the inverse of $\alpha_{u_{t_{0}}}^{i\bar{j}}$ defines a Hermitian metric on $TX$, there exists a Gauduchon metric $\hat{\alpha}$ in its conformal class. Suppose $\hat{\alpha}=e^{\mu}\alpha$. Consider the new operator 
\begin{equation*}
T(v)=e^{-\mu}\delta L(u)(v)=\frac{1}{r}\hat{\alpha}^{i\bar{j}}\partial_{i}\partial_{\bar{j}}v-\lambda r e^{-\mu} v.
\end{equation*}
Then by the maximum principle, $\mbox{Ker}(T)=\mbox{Ker}(\delta L(u))=\{0\}$. Indeed, if $x\in X$ is a maximum point of $v$--the kernel function of $T$, then $v(x)=\frac{1}{\lambda r^{2}}\alpha_{u}^{i\bar{j}}\partial_{i}\partial_{\bar{j}}v(x) \leq 0$. Similarly, at a minimum point $y$ of $v$, we have $v(y) \geq 0$. Therefore, $v = 0$. On the other hand, through direct computation, it is not difficult to obtain \begin{equation*}
T^*(v)=\frac{1}{r}\hat{\alpha}^{i\bar{j}}\partial_{i}\partial_{\bar{j}}v+\frac{n\partial v\wedge\bar{\partial}\hat{\alpha}^{n-1}}{\hat{\alpha}^{n}}-\frac{n\bar{\partial} v\wedge\partial\hat{\alpha}^{n-1}}{\hat{\alpha}^{n}}-\lambda r e^{-\mu}v.
\end{equation*}
Thus, it follows from the maximum principle that $\mbox{Ker}(T^*) = \{0\}$. Then by Theorem 2.13 in \cite{S0}, $T:C^{k+2,\alpha}(X)\to C^{k,\alpha}(X)$ is an isomorphism, and so is $\delta L(u)$.

Similar to the discussion in \cite{S0}, we have
\begin{lem}
If $L(u_{t})=\phi_{t}$ has a smooth solution for some $t>0$, then for all small $\epsilon>0$ there is a smooth function $u_{t-\epsilon}$ such that $L(u_{t-\epsilon})=\phi_{t-\epsilon}$.
\end{lem}

\subsection{Continuity method and openness: $\lambda=0$}
When $\lambda=0$. Consider the following equations for $t\in[0,1]$
\begin{equation}\label{eq:log}
\begin{split}
L(u_{t})=\phi_{t}+c_{t},
\end{split}
\end{equation}
where $c_{1}=0$. Then it has a solution for $t=1$. Let $J\subset [0,1]$ be the parameter space for which Equation \eqref{eq:log} is solvable. By Tosatti-Weinkove's \cite{TW} argument, we can show that $J$ is open in $[0,1]$.

Suppose it has a solution $u_{t_{0}}$ at $t=t_{0}$ for some constant $c_{t_{0}}$. Let $\hat{\alpha}$ be a Gauduchon metric in the conformal class of $\alpha_{u_{t_0}}$, suppose $e^{\mu}$ be its conformal factor. Choose a function $\sigma_{t_{0}}\in C^{\infty}(M)$ such that
\begin{equation*}
e^{\sigma_{t_0}}{\det}_{TX\otimes E}(\Theta_{u_{t_{0}}})^{\frac{1}{r}}=ce^{-\mu}\hat{\alpha}^{n}.
\end{equation*}
for some $c>0$ and $\int_{X}e^{\sigma_{t_0}}\det_{TX\otimes E}(\Theta_{u_{t_{0}}})^{\frac{1}{r}}=1$. Consider the following modified equations
\begin{equation}\label{eq:mlog}
\begin{split}
L(\hat{u}_{t})=\phi_{t}+\log\Big(\frac{\int_{X}e^{\sigma_{t_0}}\Omega_{\hat{u}_{t}}}{\int_{X}e^{\phi_{t}+\sigma_{t_0}}\omega^{n}}\Big),
\end{split}
\end{equation}
where $\Omega_{\hat{u}_{t}}=\det_{TX\otimes E}(\Theta_{\hat{u}_{t}})^{\frac{1}{r}}$. Since $c_{t_{0}}=\log\Big(\frac{\int_{X}e^{\sigma_{t_0}}\Omega_{u_{t_{0}}}}{\int_{X}e^{\phi_{t}+\sigma_{t_0}}\omega^{n}}\Big)$, then $u_{t_{0}}$ solves \eqref{eq:mlog} at $t=t_{0}$. It is enough to show that \eqref{eq:mlog} has a solution near $t = t_0$. Then we can take $u_{t}=\hat{u}_{t}$ and $c_{t}=\log\Big(\frac{\int_{X}e^{\sigma_{t_0}}\Omega_{\hat{u}_{t}}}{\int_{X}e^{\phi_{t}+\sigma_{t_0}}\omega^{n}}\Big)$ to obtain the solution of \eqref{eq:log} near $t=t_0$. Define
\begin{equation*}
\begin{split}
H_{k}=\{v\in C^{k,\alpha}\mid \int_{X}e^{v+\sigma_{t_0}}\Omega_{u_{t_{0}}}=1\},
\end{split}
\end{equation*}
and 
\begin{equation*}
\begin{split}
\tilde{H}_{k+2}=\{u\in C^{k+2,\alpha}\mid \Theta_{u}>_{N}0,\ \int_{X}ue^{\sigma_{t_0}}\omega^{n}=0\}.
\end{split}
\end{equation*}
Define a map $\tilde{L}:\tilde{H}_{k+2}\to H_{k}$ by
\begin{equation*}
\begin{split}
\tilde{L}(u)=\log\frac{\det_{TX\otimes E}(\Theta_{u})^{\frac{1}{r}}}{\det_{TX\otimes E}(\Theta_{u_{t_{0}}})^{\frac{1}{r}}}-\log(\int_{X}e^{\sigma_{t_0}}{\det}_{TX\otimes E}(\Theta_{u})^{\frac{1}{r}})
\end{split}
\end{equation*}
Then the equation (\ref{eq:mlog}) becomes
\begin{equation*}
\begin{split}
\tilde{L}(\hat{u}_{t})=(\phi_{t}-\phi_{t_{0}})+\Big[\log\Big(\frac{1}{\int_{X}e^{\phi_{t}+\sigma_{t_{0}}}\omega^{n}}\Big)-\log\Big(\frac{\int_{X}e^{\sigma_{t_0}}\Omega_{u_{t_{0}}}}{\int_{X}e^{\phi_{t_{0}}+\sigma_{t_0}}\omega^{n}}\Big)\Big]
\end{split}
\end{equation*}
Let $v=\tilde{L}(u)$, the tangent space at $v$ is
\begin{equation*}
T_{v}H_{k}=\{\xi\in C^{k,\alpha}\mid \int_{X}\xi e^{\sigma_{t_0}+v}\Omega_{u_{t_{0}}}=0\}.
\end{equation*}
The tangent space at $u\in \tilde{H}_{k+2}$ is
\begin{equation*}
T_{u}\tilde{H}_{k+2}=\{\eta\in C^{k+2,\alpha}\mid  \int_{X}\eta e^{\sigma_{t_0}}\omega^{n}=0\}.
\end{equation*}
Since $\tilde{L}(u_{t_{0}})=0$ and $\hat{\alpha}$ is a Gauduchon metric, the linearized operator of $\tilde{L}$ at $u_{t_{0}}$ is 
\begin{equation*}
\begin{split}
\delta\tilde{L}(u_{t_{0}}):T_{u_{t_{0}}}\tilde{H}_{k+2}&\to T_{0}H_{k}\\
\eta&\mapsto \frac{1}{r}\alpha^{i\bar{j}}_{u_{t_{0}}}\eta_{i\bar{j}}.
\end{split}
\end{equation*}
Then if $\eta\in T_{u_{t_{0}}}\tilde{H}_{k+2}$ satisfies $\delta \tilde{L}(u_{t_{0}})(\eta)=0$, it must be $0$. On the other hand, let $\xi \in T_{0}H^{k}$ and consider the following equation
\begin{equation*}
\frac{1}{r}\hat{\alpha}^{i\bar{j}}\eta_{i\bar{j}}=e^{-\mu}\xi.
\end{equation*}
Since $\int_{X}\xi e^{\sigma_{t_{0}}}\Omega_{u_{t_{0}}}=c\int_{X}\xi e^{-\mu}\hat{\alpha}^{n}=0$, the above equation has a unique solution up to a constant. In particular, there exists a unique $\eta \in T_{u_{t_{0}}}\tilde{H}_{k+2}$ such that $\delta \tilde{L}(u_{t_{0}})(\eta)= \xi$. Thus, $\delta\tilde{L}(u_{t_0})$ is surjective and hence an isomorphism. Therefore, by the implicit function theorem in Banach spaces, equation (\ref{eq:mlog}) has a solution near $t_{0}$.

\begin{lem}
If $L(u)=\phi_{t}+c_{t}$ has a smooth solution for some $t>0$ and $c_{t}$, then for all small $\epsilon>0$ there is a constant $c_{t-\epsilon}$ and a $u_{t-\epsilon}$ such that $L(u_{t-\epsilon})=\phi_{t-\epsilon}+c_{t-\epsilon}$.
\end{lem}

\subsection{Curvature tensor expressions in various basis}
Let $\omega=\sqrt{-1}g_{i\bar{j}}dz^{i}\wedge d\bar{z}^{j}$. For any point $x\in X$, suppose $g_{i\bar{j}}(x)=\delta_{ij}$. Then $\{\partial_{i}\otimes e_{\alpha}\}$ forms an orthonormal basis for $TX\otimes E$ with respect to $\omega\otimes K$. On the other hand,
we can take another orthonormal basis $\{\gamma_{a}\}_{a=1}^{nr}$ of $TX\otimes E$ with respect to metric $\omega \otimes K$, such that
\begin{equation*}
\Theta_{u}=\sum_{a}\Lambda_{a}\gamma^{a}\otimes\overline{\gamma^{a}},
\end{equation*}
where $0<\Lambda_{a}\in\mathbb{R}$ for any $a\in\{1,2\cdots,nr\}$.
Let $\gamma_{a}=\gamma_{a}^{i\bar{\alpha}}\partial_{i}\otimes e_{\alpha}$, then $\sum_{a}\gamma^{i\alpha}_{a}\overline{\gamma_{a}^{j\beta}}=\delta_{ij}\delta_{\alpha\beta}$ and $\sum_{i,\alpha}\gamma_{a}^{i\alpha}\overline{\gamma_{b}^{i\alpha}}=\delta_{ab}$. We have 
\begin{equation*}
\begin{split}
&\Theta_{u,i\alpha\overline{j\beta}}=\sum_{a}\overline{\gamma_{a}^{i\alpha}}\gamma_{a}^{j\beta}\Lambda_{a},\\
&\Lambda_{a}=\sum_{i,\alpha,j,\beta}\gamma_{a}^{i\alpha}\Theta_{u,i\alpha\overline{j\beta}}\overline{\gamma_{a}^{j\beta}},
\end{split}
\end{equation*}
and $L(u)=\log(\Pi_{a}\Lambda_{a})^{\frac{1}{r}}-\lambda r u$. Suppose $\sqrt{-1}\partial\bar{\partial}u(x)=\sum_{i}\lambda_{i}\sqrt{-1}dz^{i}\wedge d\bar{z}^{i}$, then
\begin{equation*}
\begin{split}
&\Lambda_{a}=\sum_{i,\alpha,j,\beta}\gamma_{a}^{i\alpha}\Theta_{0,i\alpha\overline{j\beta}}\overline{\gamma_{a}^{j\beta}}+\sum_{i,\alpha}|\gamma_{a}^{i\alpha}|^{2}\lambda_{i},\\
&\Theta_{0,i\alpha\overline{i\alpha}}+\lambda_{i}=\sum_{a}|\gamma_{a}^{i\alpha}|^{2}\Lambda_{a}.
\end{split}
\end{equation*}
From this equation, it can be seen that $\Lambda_{a}$ and $\lambda_{i}$ can mutually control each other.

\section{$C^{0}$-estimates and uniqueness}\label{sec:C0}

In this section, we provide $C^{0}$ estimates for the cases when $\lambda>0$ and $\lambda=0$, using different methods.

\subsection{$\lambda>0$ case}
In this case, it is easy to obtain the $C^{0}$ estimate using the maximum principle. Assume that $u$ reaches its maximum value at point $x\in X$, then $\sqrt{-1}\partial\bar{\partial}u(x)\leq 0$ and the quadratic form $\sqrt{-1}\partial\bar{\partial}u\otimes Id$ is non-positive at $x$. So
\begin{equation*}
\begin{split}
{\det}_{TX\otimes E}(\Theta_{0}+\sqrt{-1}\partial\bar{\partial}u\otimes Id_{E})(x)\leq& {\det}_{TX\otimes E}(\Theta_{0})(x).
\end{split}
\end{equation*}
 Then
 \begin{equation*}
 e^{\lambda r u(x)+\phi(x)}\leq \frac{{\det}_{TX\otimes E}(\Theta_{0})^{\frac{1}{r}}}{\omega^{n}}(x),
 \end{equation*}
which implies
\begin{equation*}
\begin{split}
\sup_{X}u\leq \frac{1}{\lambda r}(\sup_{X}\log\frac{{\det}_{TX\otimes E}(\Theta_{0})^{\frac{1}{r}}}{\omega^{n}}-\inf_{X}\phi).
\end{split}
\end{equation*}
Similarly, taking the minimum value point of $u$, we get
\begin{equation*}
\begin{split}
\inf_{X}u\geq \frac{1}{\lambda r}(\inf_{X}\log\frac{{\det}_{TX\otimes E}(\Theta_{0})^{\frac{1}{r}}}{\omega^{n}}-\sup_{X}\phi).
\end{split}
\end{equation*}
Therefore, we get the following lemma.
\begin{lem}
Let $u$ be a solution of (\ref{eq:MA}). If $\lambda>0$, then
\begin{equation*}
\begin{split}
\sup_{X}|u|\leq \frac{1}{\lambda r}(\sup_{X}\Big|\log\frac{{\det}_{TX\otimes E}(\Theta_{0})^{\frac{1}{r}}}{\omega^{n}}\Big|+\sup_{X}|\phi|).
\end{split}
\end{equation*}
\end{lem}

By using a similar method, the uniqueness of the solution can be obtained.

\begin{prop}
When $\lambda>0$, the solution of (\ref{eq:MA}) is unique.
\end{prop}

\subsection{$\lambda=0$ case} In this case, we use the Alexandrov-Bakelman-Pucci maximum principle to derive the $C^{0}$ estimate. This method was used in the Monge-Amp\`ere equation \cite{B} and later extended to general fully nonlinear equations \cite{S}.

\begin{lem}
Suppose $u$ is solution to (\ref{eq:MA}), and $\sup_{X}u=0$. Then there exists a constant $C$ depends on the background data, such that
\begin{equation*}
\sup_{X}|u|<C.
\end{equation*}
\end{lem}
\begin{proof}
By the assumption on $\sup_{X} u$, it suffices to show that $\inf_{X}u$ is bounded below. Since $\Theta_{0}+\sqrt{-1}\partial\bar{\partial}u\otimes Id>_{N}0$, then $\omega+\sqrt{-1}\partial\bar{\partial}u>0$ which implies $\Delta_{\omega}u>-2n$. By the same argument in \cite{S}, we get a uniform bound for $\|u\|^{p}$ for some $p>0$. 

Assume that $u$ attains its minimum at point $x\in X$. Choose a local coordinate chart $(z^{1},\cdots,z^{n})$ that is homeomorphic to the unit ball $B_{1}(0)$, such that $x$ corresponds to $0$. Let $v(z)=u(z)+\epsilon |z|^{2}$, then $v(0)=u(0)=\inf_{B_{1}(0)} u\leq \inf_{\partial B_{1}(0)}v-\epsilon$. Let
\begin{equation*}
P=\{x\in B_{1}(0)|\text{$|Dv|(x)<\frac{\epsilon}{2}$ and $v(y)\geq v(x)+Dv(x)\cdot(y-x)$ for $\forall y\in B_{1}(0)$}\}.
\end{equation*}
Then apply the Alexandroff-Bakelman-Pucci maximum principle (\cite[Prop 11]{S}) to $v$, we obtain
\begin{equation*}
c_{0}\epsilon^{n}\leq \int_{P}\det D^{2}v,
\end{equation*}
where $c_{0}$ only depends on $n$. As in Blocki \cite{B}, at any point $x\in P$ we have $D^{2}v(x)\geq 0$ and
\begin{equation*}
\det(D^{2}v)\leq 2^{2n} \det(v_{i\bar{j}})^{2}.
\end{equation*}
So if $x\in P$, then $D^{2}v(x)\geq 0$ implies that $u_{i\bar{j}}(x)\geq-\epsilon\delta_{i\bar{j}}$. Let $\omega_{0}=\sum_{i}\sqrt{-1}dz^{i}\wedge d\bar{z}^{i}$. Choose $\epsilon$ small enough such that $\Theta_{0}-2\epsilon\omega_{0}\otimes Id_{E}\geq\frac{1}{2} \Theta_{0}$ (depends on $\Theta_{0}$). Since the function $A\mapsto(\det A)^{1/n}$ is concave on the cone of positive hermitian $(n\times n)$-matrices, then at every point $x\in P$, we have
\begin{equation*}
\begin{split}
e^{\frac{1}{n}\phi}=&\Big(\frac{\det(\Theta_{0}-2\epsilon\omega_{0}\otimes Id_{E}+(\sqrt{-1}\partial\bar{\partial}u+2\epsilon\omega_{0})\otimes Id_{E})^{\frac{1}{r}}}{\omega^{n}}\Big)^{\frac{1}{n}}\\
\geq &\Big(\frac{\det(\Theta_{0}-2\epsilon\omega_{0}\otimes Id_{E})^{\frac{1}{r}}}{\omega^{n}}\Big)^{\frac{1}{n}}+\Big(\frac{\det((\sqrt{-1}\partial\bar{\partial}u+2\epsilon\omega_{0})\otimes Id_{E})^{\frac{1}{r}}}{\omega^{n}}\Big)^{\frac{1}{n}}\\
\geq &\Big(\frac{\det((\sqrt{-1}\partial\bar{\partial}u+2\epsilon\omega_{0})\otimes Id_{E})^{\frac{1}{r}}}{\omega^{n}}\Big)^{\frac{1}{n}}.
\end{split}
\end{equation*}
This implies $|u_{i\bar{j}}|<C$. This also gives a bound for $v_{i\bar{j}}$ at any $x\in P$. Then
\begin{equation}\label{eq:ABP}
c_{0}\epsilon^{n}\leq C^{'}\mbox{Vol}(P).
\end{equation}
By definition, for $x\in P$ we have $v(0)>v(x)-\frac{\epsilon}{2}$, and so $v(x)<L+\frac{\epsilon}{2}$. Let $L=\inf_{X} u$, then
\begin{equation*}
\mbox{Vol}(P)\leq \frac{\||v|^{p}\|_{L^{1}}}{|L+\frac{\epsilon}{2}|^{p}}.
\end{equation*}
Since $\||v|^{p}\|_{L^{1}}$ have a uniform bound, combining this with (\ref{eq:ABP}), we obtain a uniform bound for $|L|$.
\end{proof}

\begin{prop}
Suppose that for constants $c_i$ (where $i=1,2$), there exist smooth functions $u_i$ satisfying the equation $\det_{TX\otimes E}\left(\Theta_K + \sqrt{-1}\partial\bar{\partial}u_i\otimes Id_{E}\right)^{\frac{1}{r}} = e^{\phi+c_i}\omega^n$. Then $c_1 = c_2$ and $u_1 - u_2$ is a constant.
\end{prop}
\begin{proof}
If $c_1 \neq c_2$, without loss of generality, assume that $c_1 > c_2$. Let $v = u_1 - u_2$, and let $x\in X$ be the point where $v$ attains its maximum value. Then $\sqrt{-1}\partial\bar{\partial}v(x) \leq 0$, so
\begin{equation*}
e^{\phi(x)+c_1} = \frac{\det_{TX\otimes E}(\Theta_K+\sqrt{-1}\partial\bar{\partial}u_{1})^{\frac{1}{r}}}{\omega^{n}}\Big|_x \leq \frac{\det_{TX\otimes E}(\Theta_K+\sqrt{-1}\partial\bar{\partial}u_{2})^{\frac{1}{r}}}{\omega^{n}}\Big|_x = e^{\phi(x)+c_2},
\end{equation*}
which implies $c_1 \leq c_2$. This leads to a contradiction. Hence $c_1=c_2$.

By the mean value inequality, it follows that
\begin{equation*}
1+\frac{1}{nr}\alpha^{i\bar{j}}_{u_{2}}v_{i\bar{j}}=\frac{\tr_{TX\otimes E}(\Theta_{u_{2}}^{-1}\Theta_{u_{1}})}{nr}\geq \Big(\frac{\det_{TX\otimes E}(\Theta_K+\sqrt{-1}\partial\bar{\partial}u_{1})}{\det_{TX\otimes E}(\Theta_K+\sqrt{-1}\partial\bar{\partial}u_{2})}\Big)^{\frac{1}{nr}}=1.
\end{equation*}
By the maximum principle, we conclude that $v$ is a constant.
\end{proof}


\section{$C^{2}$-estimate}\label{sec:C2}
In this section, we prove the following $C^{2}$-estimate that includes a gradient term. We will adopt Hou-Ma-Wu's method \cite{HMW} to prove the following lemma.
\begin{lem}
Suppose $u$ is a solution of equation (\ref{eq:MA}), then we have
\begin{equation*}
|\partial\bar{\partial}u|\leq C(1+\sup_{X}|du|_{g}^{2})
\end{equation*}
where the constant depends on the background data, in particular $|Rm(\omega)|$, $\|\Theta_{0}\|_{C^{2}}$, $\sup|\phi|$ and $\sup|u|$.
\end{lem}

\begin{proof}
Take an orthonormal basis $\{e_{\alpha}\}_{\alpha=1}^{r}$ of $(E,K)$. Define 
\begin{equation*}
\kappa=\sup_{x\in X}\sup_{\alpha}\sup_{0\neq\xi\in T_{x}X}\frac{\Theta_{0}(\xi\otimes e_{\alpha},\xi\otimes e_{\alpha})}{|\xi|_{g}^{2}}.
\end{equation*}
Let $K=\sup_{X}(1+|du|_{g}^{2})$, $L=\sup_{X}(1+|u|)$ and $S^{1}X$ be the unit tangent vector bundle with respect to $\omega$. Then 
\begin{equation*}
W(x,\xi)=\log(\kappa+u_{i\bar{j}}\xi^{i}\overline{\xi^{j}})+\varphi(|d u|_{g}^{2})+\psi(u)
\end{equation*}
is well defined on $S^{1}X$ since $\Theta_{u}>_{N}0$. Here 
\begin{equation*}
\varphi(t)=-\log\Big(1-\frac{t}{2K}\Big)
\end{equation*}
defined on $[0,\sup |d u|_{g}^{2}]$ satisfy $\varphi^{''}=(\varphi^{'})^{2}$ and
\begin{equation*}
\frac{1}{2K}\leq\phi^{'}(t)\leq \frac{1}{K},
\end{equation*}
Similarly,  
\begin{equation*}
\psi(t)=-A\log\Big(1+\frac{t}{2L}\Big)
\end{equation*}
defined on $[\inf u,\sup u]$ for a large constant $A>0$ to be defined later. It satisfies $\frac{A}{3L}<-\psi^{'}(t)<\frac{A}{L}$ and $\psi^{''}=\frac{1}{A}(\psi^{'})^{2}$. 

Suppose $W(x,\xi)$ attains its maximum at $x_{0}\in X$ and $\xi_{0}\in T_{x_{0}}X$. We can choose a local normal coordinate $\{z^{i}\}$ near $x_{0}$, such that 
\begin{equation*}
u_{i\bar{j}}(x_{0})=u_{i\bar{i}}(x_{0})\delta_{ij},\quad g_{i\bar{j}}(x_{0})=\delta_{ij},\quad \partial_{i}g_{j\bar{k}}(x_{0})=0.
\end{equation*}

Let $\lambda_{i}=\kappa+u_{i\bar{i}}$ and assume $\lambda_{n}\leq\cdots \leq \lambda_{1}$. So we know $\xi_{0}=\partial_{1}|_{x_{0}}$. Let
\begin{equation*}
\xi=g_{1\bar{1}}^{-1}\partial_{1},
\end{equation*}
it is a smooth unit vector field defined on the neighborhood of $x_{0}$. 

The function
\begin{equation*}
\begin{split}
h(z)=\log(\kappa+g_{1\bar{1}}^{-1}u_{1\bar{1}})+\varphi(|d u|_{g}^{2})+\psi(u).
\end{split}
\end{equation*}
is well defined in a small neighborhood of $x_{0}$ and achieves its maximum at $x_{0}$. It is easy to see that
\begin{equation*}
\alpha^{i\bar{j}}_{u}=\sum_{\alpha=1}^{r}\Theta_{u}^{i\alpha\overline{j\alpha}},
\end{equation*}
is a Hermitian metric on $T^{*}X$. At the point $x_{0}$, we have
\begin{equation}\label{eq:1}
\begin{split}
\alpha_{u}^{i\overline{j}}\partial_{i}\partial_{\bar{j}}\log(\kappa+g_{1\bar{1}}^{-1}u_{1\bar{1}})=&\alpha_{u}^{i\overline{j}}\frac{\partial_{i}\partial_{\bar{j}}(g_{1\bar{1}}^{-1}u_{1\bar{1}})}{\lambda_{1}}-\alpha_{u}^{i\overline{j}}\frac{\partial_{i}(g_{1\bar{1}}^{-1}u_{1\bar{1}})\partial_{\bar{j}}(g_{1\bar{1}}^{-1}u_{1\bar{1}})}{\lambda_{1}^{2}}\\
=&\alpha_{u}^{i\overline{j}}\frac{u_{i\bar{j}1\bar{1}}}{\lambda_{1}}+\alpha_{u}^{i\overline{j}}\frac{u_{1\bar{1}}\partial_{i}\partial_{\bar{j}}g_{1\bar{1}}^{-1}}{\lambda_{1}}-\alpha_{u}^{i\overline{j}}\frac{u_{i\bar{1}1}u_{\bar{j}1\bar{1}}}{\lambda_{1}^{2}}.
\end{split}
\end{equation}

Suppose $\{\gamma_{a}^{i\alpha}\}$ is a unitary matrix with respect to $\omega\otimes K$ and $\Theta_{u}^{i\alpha\overline{j\beta}}=\sum_{a}\gamma^{j\beta}_{a}\Lambda_{a}^{-1}\overline{\gamma_{a}^{i\alpha}}$ where $\Lambda_{1}\geq\cdots\geq\Lambda_{rn}>0$ are eigenvalues of $\Theta_{u}$.
For any $\{x_{i\bar{j}}\}\in M^{n\times n}$, we have
\begin{equation}
\begin{split}
\sum_{\alpha,\beta}\Theta^{i\alpha\overline{j\beta}}_{u}\Theta_{u}^{k\beta\overline{l\alpha}}x_{k\bar{j}}\overline{x_{l\bar{i}}}=&\sum_{\alpha,\beta}\gamma^{j\beta}_{a}\Lambda_{a}^{-1}\overline{\gamma^{i\alpha}_{a}}\gamma^{l\alpha}_{b}\Lambda_{b}^{-1}\overline{\gamma^{k\beta}_{b}}x_{k\bar{j}}\overline{x_{l\bar{i}}}\\
=&\Lambda_{a}^{-1}\Lambda_{b}^{-1}(\overline{\sum_{\alpha}\gamma^{i\alpha}_{a}\overline{\gamma^{l\alpha}_{b}}x_{l\bar{i}}})\sum_{\beta}(\gamma^{j\beta}_{a}\overline{\gamma^{k\beta}_{b}}x_{k\bar{j}})\\
\geq& 0
\end{split}
\end{equation}
and
\begin{equation*}
\Lambda_{1}=\gamma_{1}^{j\beta}\overline{\gamma_{1}^{i\alpha}}\Theta_{u,i\alpha\overline{j\beta}}=(\gamma_{1}^{j\beta}\overline{\gamma_{1}^{i\alpha}}\Theta_{0,i\alpha\overline{j\beta}}-\kappa)+\sum_{i,\alpha}|\gamma_{1}^{i\alpha}|^{2}(u_{i\bar{i}}+\kappa)\leq \lambda_{1},
\end{equation*}
here we use $\gamma_{1}^{j\beta}\overline{\gamma_{1}^{i\alpha}}\Theta_{0,i\alpha\overline{j\beta}}\leq \kappa$.

Using equation (\ref{eq:MA}), at the point $x_{0}$ we have
\begin{equation}\label{eq:eq1}
\begin{split}
\lambda r u_{i}+\phi_{i}=&\partial_{i}\log\frac{\det(\Theta_{u})^{\frac{1}{r}}}{\omega^{n}}\\
=&\frac{1}{r}\Theta_{u}^{k\alpha\overline{l\beta}}\partial_{i}\Theta_{u,k\alpha\overline{l\beta}}-\partial_{i}\log\det(g_{i\bar{j}})\\
=&\frac{1}{r}\Theta_{u}^{k\alpha\overline{l\beta}}\partial_{i}\Theta_{0,k\alpha\overline{l\beta}}+\frac{1}{r}\alpha_{u}^{k\overline{l}}u_{ik\bar{l}}
\end{split}
\end{equation}
and
\begin{equation}\label{eq:eq2}
\begin{split}
r\partial_{1}\partial_{\bar{1}}\log\det(\Theta_{u})^{\frac{1}{r}}=&-\Theta_{u}^{i\alpha\overline{j\beta}}\partial_{1}\Theta_{u,k\gamma\overline{j\beta}}\Theta_{u}^{k\gamma\overline{l\delta}}\partial_{\bar{1}}\Theta_{u,i\alpha\overline{l\delta}}+\Theta_{u}^{i\alpha\overline{j\beta}}\partial_{1}\partial_{\bar{1}}\Theta_{u,i\alpha\overline{j\beta}}\\
=&-\Theta_{u}^{i\alpha\overline{j\beta}}\partial_{1}\Theta_{u,k\gamma\overline{j\beta}}\Theta_{u}^{k\gamma\overline{l\delta}}\partial_{\bar{1}}\Theta_{u,i\alpha\overline{l\delta}}+\Theta_{u}^{i\alpha\overline{j\beta}}\partial_{1}\partial_{\bar{1}}\Theta_{0,i\alpha\overline{j\beta}}+\alpha_{u}^{i\overline{j}}u_{1\overline{1}i\overline{j}}\\
\leq&-\frac{(1-\epsilon_{1})}{\lambda_{1}}\sum_{i}\alpha_{u}^{k\bar{l}}u_{1\bar{i}k}u_{\bar{1}i\bar{l}}+\frac{\frac{1}{\epsilon_{1}}-1}{\lambda_{1}}\sum_{i,\alpha}\Theta_{u}^{k\gamma\overline{l\delta}}\partial_{1}\Theta_{0,k\gamma\overline{i\alpha}}\partial_{\bar{1}}\Theta_{0,i\alpha\overline{l\delta}}\\
&+\Theta_{u}^{i\alpha\overline{j\beta}}\partial_{1}\partial_{\bar{1}}\Theta_{0,i\alpha\overline{j\beta}}+\alpha_{u}^{i\overline{j}}u_{1\overline{1}i\overline{j}},
\end{split}
\end{equation}
$\epsilon_{1}$ will be determined later. Here we have used
\begin{equation*}
\begin{split}
\Theta_{u}^{i\alpha\overline{j\beta}}\partial_{1}\Theta_{u,k\gamma\overline{j\beta}}\Theta_{u}^{k\gamma\overline{l\delta}}\partial_{\bar{1}}\Theta_{u,i\alpha\overline{l\delta}}=&\Lambda_{a}^{-1}\Lambda_{b}^{-1}(\overline{\gamma_{b}^{k\gamma}}\gamma_{a}^{j\beta}\partial_{1}\Theta_{u,k\gamma\overline{j\beta}})\overline{\gamma^{i\alpha}_{a}\overline{\gamma_{b}^{l\delta}}\partial_{1}\Theta_{u,l\delta\overline{i\alpha}}}\\
\geq&\frac{1}{\Lambda_{1}}\sum_{i,\alpha}\Theta_{u}^{k\gamma\overline{l\delta}}\partial_{1}\Theta_{u,k\gamma\overline{i\alpha}}\partial_{\bar{1}}\Theta_{u,i\alpha\overline{l\delta}}\\
\geq&\frac{(1-\epsilon_{1})}{\lambda_{1}}\sum_{i}\alpha_{u}^{k\bar{l}}u_{1\bar{i}k}u_{\bar{1}i\bar{l}}-\frac{\frac{1}{\epsilon_{1}}-1}{\lambda_{1}}\sum_{i,\alpha}\Theta_{u}^{k\gamma\overline{l\delta}}\partial_{1}\Theta_{0,k\gamma\overline{i\alpha}}\partial_{\bar{1}}\Theta_{0,i\alpha\overline{l\delta}}.
\end{split}
\end{equation*}

For $\varphi(|du|_{g}^{2})$, we have
\begin{equation}\label{eq:2}
\begin{split}
\alpha_{u}^{i\overline{j}}\partial_{i}\partial_{\bar{j}}\varphi(|d u|_{g}^{2})=&\alpha_{u}^{i\overline{j}}\varphi^{'}\partial_{i}\partial_{\bar{j}}(|d u|_{g}^{2})+\alpha_{u}^{i\overline{j}}\varphi^{''}\partial_{i}(|d u|_{g}^{2})\partial_{\bar{j}}(|d u|_{g}^{2})\\
= &2\alpha_{u}^{i\overline{j}}\varphi^{'}\sum_{k}({R_{i\bar{j}}}^{k\bar{l}}u_{k}u_{\bar{l}}+2\Re(u_{i\bar{j}k}u_{\bar{k}}))\\
&+2\varphi^{'}\alpha_{u}^{i\bar{i}}(|u_{i\bar{i}}|^{2}+|u_{ii}|^{2})+\varphi^{''}|\partial|d u|_{g}^{2}|_{u}^{2},
\end{split}
\end{equation}
where $|\partial f|_{u}^{2}=\alpha_{u}^{i\bar{j}}f_{i}f_{\bar{j}}$.

For $\psi(u)$, we have
\begin{equation}\label{eq:4}
\begin{split}
\alpha_{u}^{i\bar{j}}\partial_{i}\partial_{\bar{j}}\psi(u)=&\alpha_{u}^{i\bar{j}}\psi^{'}u_{i\bar{j}}+\alpha_{u}^{i\bar{j}}\psi^{''}u_{i}u_{\bar{j}}.
\end{split}
\end{equation}

Since $h(z)$ attains its maximum at $x_{0}\in X$, then
\begin{equation}\label{eq:c1}
\begin{split}
\frac{\partial_{i}(g_{1\overline{1}}^{-1}u_{1\bar{1}})}{\lambda_{1}}+\varphi^{'}\partial_{i}|\partial u|_{0}^{2}+\psi^{'}u_{i}=0.
\end{split}
\end{equation}

Let $\mathcal{S}=\sum_{i,\alpha}\Theta_{u}^{i\alpha\overline{i\alpha}}=\sum_{a}\frac{1}{\Lambda_{a}}$. Since 
\begin{equation*}
\begin{split}
|\Theta^{i\alpha\overline{j\beta}}|=&|\sum_{a}\gamma_{a}^{j\beta}\Lambda_{a}^{-1}\overline{\gamma_{a}^{i\alpha}}|\leq C\mathcal{S},
\end{split}
\end{equation*}
then combining (\ref{eq:1}), (\ref{eq:eq1}), (\ref{eq:eq2}), (\ref{eq:2}), (\ref{eq:4}) we get
\begin{equation}\label{eq:good}
\begin{split}
0\geq &\alpha_{u}^{i\bar{j}}h_{i\bar{j}}-\frac{r\partial_{1}\partial_{\bar{1}}\log\det(\Theta_{u})^{\frac{1}{r}}}{\lambda_{1}}+\frac{\lambda r^{2}u_{1\bar{1}}+r\phi_{1\bar{1}}+r(\log\det(g_{p\bar{q}})_{1\bar{1}}}{\lambda_{1}}\\
\geq&-C_{1}\mathcal{S}-C_{2}\frac{\mathcal{S}}{\lambda_{1}}-\frac{|\partial(g_{1\bar{1}}^{-1}u_{1\bar{1}})|^{2}_{u}}{\lambda_{1}^{2}}+\frac{(1-\epsilon_{1})}{\lambda_{1}^{2}}\sum_{k}\alpha_{u}^{i\bar{j}}u_{1\bar{k}i}u_{\bar{1}k\bar{j}}\\
&-\frac{C_{3}(\frac{1}{\epsilon_{1}}-1)}{\lambda_{1}^{2}}\mathcal{S}+\varphi^{''}|\partial|d u|_{g}^{2}|^{2}_{u}+2\varphi^{'}\alpha_{u}^{i\bar{i}}|u_{i\bar{i}}|^{2}\\
&+4\lambda r^{2}\varphi^{'}|\partial u|_{g}^{2}+4r\varphi^{'}\sum_{k}\phi_{k}u_{\bar{k}}-4\varphi^{'} \Theta_{u}^{i\alpha\overline{j\beta}}\partial_{k}\Theta_{0,i\alpha\overline{j\beta}}u_{\bar{k}}\\
&-C_{4}\varphi^{'}\mathcal{S}|d u|_{g}^{2}+\psi^{'}\alpha_{u}^{i\bar{j}}u_{i\bar{j}}+\psi^{''}|\partial u|_{u}^{2}\\
&+\frac{\lambda r^{2}u_{1\bar{1}}+r\phi_{1\bar{1}}+r\log\det(g_{p\bar{q}})_{1\bar{1}}}{\lambda_{1}}.
\end{split}
\end{equation}
We also have
\begin{equation*}
\begin{split}
\mathcal{S}=\sum_{a}\Lambda_{a}^{-1}\geq \frac{nr}{(\Pi_{a}\Lambda_{a})^{rn}}>C_{5},
\end{split}
\end{equation*}
and
\begin{equation*}
\begin{split}
-\alpha^{i\bar{j}}_{u}u_{i\bar{j}}=\tr(\Theta_{u}^{-1}\Theta_{0})-nr\geq C_{6}\mathcal{S}-nr.
\end{split}
\end{equation*}
All the constants $C_{i}$ above depend only on $A_{0}$, $\omega$ $n$ and $r$.

By (\ref{eq:c1}) we obtain
\begin{equation}
\begin{split}
(\varphi^{'})^{2}|\partial|d u|_{g}^{2}|_{u}^{2}=&\Big|\frac{\partial(g_{1\bar{1}}^{-1} u_{1\bar{1}})}{\lambda_{1}}+\psi^{'}\partial u\Big|_{u}^{2}
\\
\geq&\delta \frac{|\partial(g_{1\bar{1}}^{-1}u_{1\bar{1}})|_{u}^{2}}{\lambda_{1}^{2}}-\frac{\delta(\psi^{'})^{2}}{1-\delta}|\partial u|_{u}^{2}.
\end{split}
\end{equation}
Suppose $\delta$ is small enough such that
\begin{equation*}
\frac{\delta(\psi^{'})^{2}}{1-\delta}\leq \psi^{''}.
\end{equation*}
In fact, we can take $\delta\leq\frac{1}{A+1}$. We also know that
\begin{equation}
\begin{split}
(1-\delta) |\partial(g_{1\bar{1}}^{-1}u_{1\bar{1}})|_{u}^{2}= &(1-\delta)\alpha_{u}^{i\bar{j}}u_{1\bar{1}i}u_{1\bar{1}\bar{j}}.
\end{split}
\end{equation}
Choose $\epsilon_{1}\leq\delta$, then 
\begin{equation*}
-\frac{|\partial(g_{1\bar{1}}^{-1}u_{1\bar{1}})|^{2}_{u}}{\lambda_{1}^{2}}+\frac{(1-\epsilon_{1})}{\lambda_{1}^{2}}\sum_{k}\alpha_{u}^{i\bar{j}}u_{1\bar{k}i}u_{\bar{1}k\bar{j}}+\varphi^{''}|\partial|d u|_{g}^{2}|^{2}_{u}\geq -\psi^{''}|\partial u|_{u}^{2}.
\end{equation*}
Therefore by (\ref{eq:good}), we get
\begin{equation*}
\begin{split}
0\geq &\frac{1}{K}\alpha^{i\bar{i}}_{u}|u_{i\bar{i}}|^{2}+C_{6}\frac{A}{3L}\mathcal{S}-\frac{C_{3}(\frac{1}{\epsilon_{1}}-1)}{\lambda_{1}^{2}}\mathcal{S}-C_{7}\frac{\mathcal{S}}{\lambda_{1}}-C_{8}K^{-\frac{1}{2}}\mathcal{S}\\
&-C_{9}\mathcal{S}-C_{10}K^{-\frac{1}{2}}-C_{11}-\frac{C_{12}}{\lambda_{1}}-\frac{nrA}{L}
\end{split}
\end{equation*}
We may assume $\lambda_{1}\geq\sqrt{\frac{6LC_{3}}{C_{6}}}$, otherwise we get a bounded of $\lambda_{1}$. Then we can take $A$ large enough such that
\begin{equation*}
(\frac{C_{6}}{3L}-\frac{C_{3}}{\lambda_{1}^{2}})A\geq\frac{C_{7}}{\lambda_{1}}+C_{8}K^{-\frac{1}{2}}+C_{9}.
\end{equation*}
Let
\begin{equation*}
A=\frac{\sqrt{6L}C_{7}}{\sqrt{C_{6}C_{3}}}+\frac{6L}{C_{6}}C_{8}K^{-\frac{1}{2}}+\frac{6LC_{9}}{C_{6}},
\end{equation*}
and $\delta=\epsilon_{1}=\frac{1}{1+A}$. Then
\begin{equation*}
C_{6}\frac{A}{3L}\mathcal{S}-\frac{C_{3}(\frac{1}{\epsilon_{1}}-1)}{\lambda_{1}^{2}}\mathcal{S}-C_{7}\frac{\mathcal{S}}{\lambda_{1}}-C_{8}K^{-\frac{1}{2}}\mathcal{S}-C_{9}\mathcal{S}\geq 0,
\end{equation*}
and
\begin{equation*}
\begin{split}
0\geq &\frac{1}{K}\alpha^{i\bar{i}}_{u}|u_{i\bar{i}}|^{2}-C_{13}K^{-\frac{1}{2}}-C_{14}\\
\geq&\frac{1}{K}\alpha^{1\bar{1}}_{u}|u_{1\bar{1}}|^{2}-C_{13}K^{-\frac{1}{2}}-C_{15}\\
\geq& \frac{1}{K}\lambda_{1}-C_{13}K^{-\frac{1}{2}}-C_{15},
\end{split}
\end{equation*}
where $C_{13},C_{14}$ and $C_{15}$ only depends on the background data and $\sup|\phi|$, $\sup|u|$. This implies a bound of $\lambda_{1}$.
\end{proof}

\section{$C^{1}$-estimates}\label{sec:C1}
In this section, we will use the blow-up argument to provide an estimate for the gradient of the solution to the equation. This method has been used by Chen \cite{Ch} for the Monge-Amp\`ere equation and later by Dinew and Kolodziej \cite{DK} for Hessian equations.

\begin{thm}
Let $u$ be a solution to (\ref{eq:MA}), then there is a constant $C>0$ depends on $\sup_{X}|\phi|$, such that
\begin{equation*}
\sup_{X}|d u|_{g}\leq C.
\end{equation*}
\end{thm}
\begin{proof}
We will prove this by contradiction. Assume $\phi_{j}$ is a sequence of smooth functions with $\sup_{X}|\phi_{j}|<C_{1}$ and $u_{i}$ is the solution of 
\begin{equation*}
{\det}_{TX\otimes E}(\Theta_{0}+\sqrt{-1}\partial\bar{\partial}u_{j}\otimes Id_{E})^{\frac{1}{r}}=e^{\lambda ru_{j}+\phi_{j}}\omega^{n},
\end{equation*}
satisfies $\sup_{X}|d u_{j}|_{0}=l_{j}\to+\infty$. Let $x_{j}\in X$ be the point such that $|d u_{j}|_{0}(x_{j})=l_{j}$.  Since $X$ is compact, there exists a cluster point $x_{0}$ for $\{x_{j}\}$. Without loss of generality, let $\lim_{j\to\infty}x_{j}= x_{0}$. Let $B_{2}(0)\subset\mathbb{C}^{n}$ be a coordinate chart centered at $x_{0}$ with $g_{i\bar{j}}(x_{0})=\delta_{i\bar{j}}$. For $j$ large enough, we may assume $x_{i}\in B_{1}(0)$. Define
\begin{equation*}
\hat{u}_{j}(z)=u_{j}(x_{j}+\frac{1}{l_{j}}z)\quad\quad \forall z\in B_{l_{j}}(0).
\end{equation*}
Then by $C^{0}$ estimates and $|\partial\bar{\partial}u|\leq C(1+\sup_{X}|du|_{g}^{2})$, we have 
\begin{equation*}
\begin{split}
&\sup_{B_{l_{j}}(0)}|\hat{u}_{j}|\leq C,\quad\quad \sup_{B_{l_{j}}(0)}|d\hat{u}_{j}|<C,\\
&|d\hat{u}_{j}|(0)=1,\quad\quad \sup_{B_{l_{j}}(0)}|\partial\bar{\partial}\hat{u}_{j}|<C.
\end{split}
\end{equation*}
By Sobolev embedding theorem, for any $p > 1$, $\|\hat{u}_{j}\|_{W^{2,p}_{loc}}$ and $\|\hat{u}_{j}\|_{C^{1,\alpha}_{loc}}$ are uniformly bounded. So there exists a subsequence of $\hat{u}_{j}$ that converges in $W^{2,p}_{loc}$ and $C^{1,\alpha}_{loc}$ topologies, to a function $u$ in $\mathbb{C}^{n}$. And
\begin{equation*}
\sup_{\mathbb{C}^{n}}|u|+\sup_{\mathbb{C}^{n}}|du|<C,\quad |d u|(0)\neq 0.
\end{equation*}
On $B_{l_{j}}(0)$, we have
\begin{equation*}
\omega+l_{j}^{2}\sqrt{-1}\partial\bar{\partial}\hat{u}_{j}>0
\end{equation*}
are positive $(1,1)$-forms. These inequalities tell us that the limiting function $u$ is psh. Let $j\to\infty$, then
\begin{equation*}
{\det}_{TX\otimes E}(l_{j}^{-2}\Theta_{0}+\sqrt{-1}\partial\bar{\partial}\hat{u}_{j}\otimes Id_{E})^{\frac{1}{r}}=l_{j}^{-2n}e^{\lambda ru_{j}+\phi_{j}}\omega^{n}\to 0.
\end{equation*}
Then this fact can be read also in the pluripotential sense and thus one can extract the weak limit satisfying
\begin{equation*}
{\det}_{TX\otimes E}(\sqrt{-1}\partial\bar{\partial}u\otimes Id_{E})^{\frac{1}{r}}={\det}_{TX}(\sqrt{-1}\partial\bar{\partial}u)=0.
\end{equation*}
Thus, $u$ is a psh maximal function. A Liouville-type theorem (see \cite[Theorem 3.2]{DK}) states that any bounded, psh maximal function $u$ in $\mathbb{C}^n$ with bounded gradient is constant. This contradicts $|du|(0) \neq 0$.
\end{proof}

\section{$C^{2,\alpha}$-estimate}\label{sec:EK}
In this section, based on the complex version of the Evans-Krylov estimate proposed by Trudinger \cite{Tr}, we present a complete proof of the $C^{2,\alpha}$ estimate for the function $u$ (for additional reference, see also \cite{Siu, TW}).

We consider the following equation
\begin{equation}\label{eq:loc}
\log\det(\Theta_{u})^{\frac{1}{r}} = \psi
\end{equation}
on an open set $\Omega$ in $\mathbb{C}^n$ that contains the ball $B_{2R}$ of radius $2R$. Here $\psi=\lambda r u+\phi+\log\det g_{i\bar{j}}$.  Let $\gamma\in \mathbb{C}^{n}$ be an arbitrary vector. Differentiating equation \eqref{eq:loc} with respect to $\gamma$ and then $\bar{\gamma}$ we obtain
\begin{equation*}
\frac{1}{r}\tr_{TX\otimes E}(\Theta_{u}^{-1}(\Theta_{u})_{\gamma})=\psi_{\gamma}
\end{equation*}
and
\begin{equation*}
\frac{1}{r}\tr_{TX\otimes E}(\Theta_{u}^{-1}(\Theta_{u})_{\gamma\bar{\gamma}})-\frac{1}{r}\tr_{TX\otimes E}(\Theta_{u}^{-1}(\Theta_{u})_{\bar{\gamma}}\Theta_{u}^{-1}(\Theta_{u})_{\gamma})=\psi_{\gamma\bar{\gamma}}.
\end{equation*}
The second term on the left-hand side of the equation is non-positive, and hence we have 
\begin{equation*}
\frac{1}{r}\alpha^{i\bar{j}}_{u}u_{i\bar{j}\gamma\bar{\gamma}} \geq \psi_{\gamma\bar{\gamma}}-\frac{1}{r}\tr_{TX\otimes E}(\Theta_{u}^{-1}\Theta_{0}).
\end{equation*}
Let $w=u_{\gamma\bar{\gamma}}$ and $h=\psi_{\gamma\bar{\gamma}}-\frac{1}{r}\tr_{TX\otimes E}(\Theta_{u}^{-1}\Theta_{0})$, we then have 
\begin{equation}\label{eq:sub}
\frac{1}{r}\alpha^{i\bar{j}}_{u}w_{i\bar{j}} \geq -C_{1},
\end{equation}
where $C_{1}$ depends on $\sup_{\Omega}|\partial\bar{\partial}u|$.

On the other hand, since $\log \det A$ is a concave function defined on the space of positive definite matrices, we know that for any two positive definite matrices $A$ and $B$, the inequality $\log \det A \leq \log \det B + \tr (B^{-1}(A-B))$ holds. Hence for any $x,y\in U$, we have
\begin{equation*}
\log\det (\Theta_{u}(x))^{\frac{1}{r}} \leq \log\det (\Theta_{u}(y))^{\frac{1}{r}}+\frac{1}{r}\tr_{TX\otimes E} (\Theta_{u}^{-1}(y)(\Theta_{u}(x)-\Theta_{u}(y))).
\end{equation*}
It follows that
\begin{equation*}
\tr_{TX\otimes E} (\Theta_{u}^{-1}(y)(\Theta_{u}(y)-\Theta_{u}(x)))\leq r(\psi(y)-\psi(x))\leq C_{2} R,
\end{equation*}
where $C_{2}$ depends on $\sup_{\Omega}|du|$. Since we have a priori estimates on $\sqrt{-1}\partial\bar{\partial u}$ and hence on $\Theta_{u}^{-1}$, by a lemma from linear algebra \cite{Siu}, we can find unit vectors $\widetilde{\gamma}_{1},\cdots,\widetilde{\gamma}_{N}\in \mathbb{C}^{n\times r}$ and real-valued functions $\beta_{1},\cdots,\beta_{N}$ satisfying
\begin{equation*}
0<\frac{1}{C^{*}}\leq \beta_{\nu}\leq C^{*}, \quad \text{for $\nu=1,\cdots, N$, }
\end{equation*}
such that
\begin{equation*}
\Theta_{u}^{i\alpha\overline{j\beta}}(y)=\sum_{\nu=1}^{N}\beta_{\nu}(y)\widetilde{\gamma}_{\nu}^{i\alpha}\overline{\widetilde{\gamma}_{\nu}^{j\beta}}.
\end{equation*}
Let $A_{\nu}^{i\bar{j}}=\sum_{\alpha}(\widetilde{\gamma}_{\nu})^{i\alpha}\overline{(\widetilde{\gamma}_{\nu})^{j\alpha}}$. For any $\nu$, since $\tr(A_\nu) = 1$ and $A_{\nu}$ is positive semi-definite, there exist finitely many positive eigenvalues $\lambda_{\nu,1}, \cdots, \lambda_{\nu,k_\nu}$, with the corresponding eigenvectors $\gamma_{\nu,1}, \cdots, \gamma_{\nu,k_\nu}$. Let $w_{\nu,l}=\gamma_{\nu,l}^{i}\overline{\gamma_{\nu,l}^{j}}u_{i\bar{j}}$ and $w_{\nu}=A_{v}^{i\bar{j}}u_{i\bar{j}}=\sum_{l=1}^{k_{\nu}}\lambda_{\nu,l}w_{\nu,l}$. Then 
\begin{equation*}
\tr_{TX\otimes E} (\Theta_{u}^{-1}(y)(\Theta_{u}(y)-\Theta_{u}(x)))=\sum_{\nu=1}^{N}\beta_{\nu}(y)(w_{\nu}(y)-w_{\nu}(x))+\tr_{TX\otimes E} (\Theta_{u}^{-1}(y)(\Theta_{0}(y)-\Theta_{0}(x))),
\end{equation*}
and 
\begin{equation}\label{eq:ek1}
\sum_{\nu}\sum_{l=1}^{k_{\nu}}\beta_{\nu}(y)\lambda_{\nu,l}(w_{v,l}(y)-w_{\nu,l}(x))=\sum_{\nu=1}^{N}\beta_{\nu}(y)(w_{\nu}(y)-w_{\nu}(x))\leq C_{3}R.
\end{equation}
Similarly, for any unitary matrix $U \in U(n, \mathbb{C})$, setting $\prescript{U}{}{\widetilde{\gamma}}_{\nu}^{i\alpha} = \sum_{j}U^{ij}\widetilde{\gamma}_{\nu}^{j\alpha}$, there also exist $\frac{1}{C^{*}}\leq \prescript{U}{}{\beta}_{1},\cdots,\prescript{U}{}{\beta}_{N}\leq C^{*}$ such that
\begin{equation*}
\Theta_{u}^{i\alpha\overline{j\beta}}(y)=\sum_{\nu=1}^{N}\prescript{U}{}{\beta}_{\nu}(y)\prescript{U}{}{\widetilde{\gamma}}_{\nu}^{i\alpha}\overline{\prescript{U}{}{\widetilde{\gamma}}_{\nu}^{j\beta}}.
\end{equation*}
Then $\prescript{U}{}{A}_{\nu}=UA_{\nu}U^{*}$. Thus, for any unit vector $\gamma \in \mathbb{C}^n$, we may assume that $\gamma_{1,1} = \gamma$. To proceed, we need the following lemma.
\begin{lem}\cite[Thm. 9.22]{GT}\label{l:H}
Let $g$ be a Hermitian metric on $\Omega\subset\mathbb{C}^{n}$ which is uniformly equivalent to the Euclidean metric. Suppose that $v\geq 0$ satisfies
\begin{equation*}
g^{i\bar{j}}v_{i\bar{j}}\leq \theta,
\end{equation*}
on $B_{2R}\subset\Omega$. Then there exist uniform constants $p>0$ and $C>0$ such that
\begin{equation*}
\Big(\frac{1}{R^{2n}}\int_{B_{R}}v^{p}\Big)^{\frac{1}{p}}\leq C(\inf_{B_{R}}v+R\|\theta\|_{L^{2n}(B_{2R})}).
\end{equation*}
\end{lem}

For $s=1,2$, write
\begin{equation*}
\begin{split}
M_{s,\nu,l}=\sup_{B_{sR}} w_{\nu,l}, \quad m_{s,\nu,l}=\inf_{B_{sR}}w_{\nu,l}, \quad \omega(sR)=\sum_{\nu}\sum_{l=1}^{k_{\nu}}(M_{s,\nu,l}-m_{s,\nu,l}).
\end{split}
\end{equation*}
Since each $w_{\nu,l}$ satisfies \eqref{eq:sub}. we can apply Lemma \ref{l:H} to $M_{2,\nu,l}-w_{\nu,l}$ to obtain
\begin{equation}\label{eq:ek3}
\Big(\frac{1}{R^{2n}}\int_{B_{R}}(M_{2,\nu,l}-w_{\nu,l})^{p}\Big)^{\frac{1}{p}}\leq C(M_{2,\nu,l}-M_{1,\nu,l}+R^2).
\end{equation}
Thus for a fixed $(\nu_{0},l_{0})$ we have
\begin{equation}\label{eq:ek2}
\begin{split}
\Big(\frac{1}{R^{2n}}\int_{B_{R}}(\sum_{(\nu,l)\neq (\nu_{0},l_{0})}(M_{2,\nu,l}-w_{\nu,l}))^{p}\Big)^{\frac{1}{p}}\leq&M^{\frac{1}{p}}\sum_{(\nu,l)\neq(\nu_{0},l_{0})}\Big(\frac{1}{R^{2n}}\int_{B_{R}}(M_{2,\nu,l}-w_{\nu,l})^{p}\Big)^{\frac{1}{p}}\\
\leq& C(\sum_{(\nu,l)\neq (\nu_{0},l_{0})}(M_{2,\nu,l}-M_{1,\nu,l})+R^2)\\
\leq &C(\omega(2R)-\omega(R)+R^{2}),
\end{split}
\end{equation}
since 
\begin{equation*}
(M_{2,\nu,l}-m_{2,\nu,l})-(M_{1,\nu,l}-m_{1,\nu,l})=(M_{2,\nu,l}-M_{1,\nu,l})+(m_{1,\nu,l}-m_{2,\nu,l})\geq M_{2,\nu,l}-M_{1,\nu,l}.
\end{equation*}
From \eqref{eq:ek1} we have
\begin{equation*}
\beta_{\nu_{0}}(y)\lambda_{\nu_{0},l_{0}}(w_{\nu_{0},l_{0}}(y)-w_{\nu_{0},l_{0}}(x))\leq CR+\sum_{(\nu,l)\neq(\nu_{0},l_{0})}\beta_{\nu}(y)\lambda_{\nu,l}(w_{\nu,l}(x)-w_{\nu,l}(y)).
\end{equation*}
Hence by choosing $x$ so that $w_{\nu_{0},l_{0}}(x)$ approaches $m_{2,\nu_{0},l_{0}}$ and using the mean value theorem we have
\begin{equation*}
w_{\nu_{0},l_{0}}(y)-m_{2,\nu_{0},l_{0}}\leq C\big(R+\sum_{(\nu,l)\neq(\nu_{0},l_{0})}(M_{2,\nu,l}-w_{\nu,l}(y))\big).
\end{equation*}
Integrating in $y$ over $B_{R}$ and applying \eqref{eq:ek2} we have
\begin{equation}\label{eq:ek4}
\Big(\frac{1}{R^{2n}}\int_{B_{R}}(w_{\nu_{0},l_{0}}-m_{2,\nu_{0},l_{0}})^{p}\Big)^{\frac{1}{p}}\leq C(\omega(2R)-\Omega(R)+R).
\end{equation}
Adding it to \eqref{eq:ek3} and summing over $(\nu, l)$, we obtain 
\begin{equation*}
\omega(2R)\leq C(\omega(2R)-\omega(R)+R),
\end{equation*}
hence
\begin{equation}
\omega(R)\leq \delta\omega(R)+R,
\end{equation}
where $0<\delta=1-\frac{1}{C}<1$. Then by a standard argument (\cite[Lem. 8.23]{GT}) that there exist uniform constants $C$ and $\kappa>0$ such that
\begin{equation*}
\omega(R)\leq CR^{\kappa}.
\end{equation*}
Thus, for any unit vector $\gamma \in \mathbb{C}^n$ and $x, y \in B_R$, we have $|u_{\gamma\bar{\gamma}}(x) - u_{\gamma\bar{\gamma}}(y)| \leq C |x - y|^\kappa$. This yields the Hölder estimate for $\sqrt{-1}\partial\bar{\partial}u$.

\medskip

\statementtitle{Acknowledgement} 
The author would like to thank the referees for their helpful comments. The author are partially supported by NSFC (Grant No. 12141104) and Natural Science Foundation of Jiangsu Province, China (Grant No. BK20241434).

\statementtitle{Conflict of Interest Statement}
The author declare no conflict of interest.

\statementtitle{Data Availability Statement}
No datasets were generated or analyzed in this theoretical study.


\end{document}